\documentclass[12pt]{article}
\usepackage{epsf}

\usepackage{amsthm}
\usepackage{amssymb}
\usepackage{amsmath}

\usepackage[top=1.5cm,bottom=1.5cm,left=2.5cm,right=2.5cm,includehead,includefoot
           ]{geometry}

\makeatletter

\numberwithin{equation}{section}
\newtheorem{theorem}{Theorem}[section]
\newtheorem{lemma}[theorem]{Lemma}

\title{Local Lie derivations on von Neumann algebras and algebras of locally measurable operators}
\author{\begin{tabular}{c} Jun He$^{1}$ and Guangyu An$^{2}$\footnote{Corresponding author.
E-mail address: anguangyu310@163.com}
\\{\small\it  $^{1}$Department of Mathematics, Anhui Polytechnic University}\\
{\small\it Wuhu 241000, China}
\\{\small\it $^{2}$Department of Mathematics, Shaanxi University of Science and Technology}\\
{\small\it Xi'an 710021, China}
\end{tabular}}
\date{}
\begin{document}
\maketitle \abstract
Let $\mathcal{A}$ be a unital associative algebra and $\mathcal{M}$ be an $\mathcal{A}$-bimodule.
A linear mapping $\varphi$ from $\mathcal{A}$ into an $\mathcal{A}$-bimodule $\mathcal{M}$ is called a Lie derivation if
$\varphi[A,B]=[\varphi(A),B]+[A,\varphi(B)]$ for each $A,B$ in $\mathcal{A}$,
and $\varphi$ is called a \emph{local Lie derivation} if for every $A$ in $\mathcal{A}$,
there exists a Lie derivation $\varphi_{A}$ (depending on $A$) from $\mathcal{A}$ into $\mathcal{M}$
such that $\varphi(A)=\varphi_{A}(A)$. In this paper, we prove that every local Lie derivation on von Neumann algebras is a Lie derivation;
and we show that if $\mathcal M$ is a type I von Neumann algebra with atomic lattice
of projections, then every local Lie derivation on $LS(\mathcal M)$ is a Lie derivation.

\

\

\
\section{Introduction}\

Let $\mathcal{A}$ be a unital associative algebra
over the complex field $\mathbb{C}$ and $\mathcal{M}$ be an $\mathcal{A}$-bimodule.
An linear mapping $\delta$ from $\mathcal{A}$ into $\mathcal{M}$ is
called a \emph{derivation} if $\delta(AB)=\delta(A)B+A\delta(B)$ for each $A$ and $B$ in $\mathcal{A}$.
In particular, a derivation $\delta_M$ defined by $\delta_{M}(A)=MA-AM$ for every $A$ in $\mathcal{A}$ is called an inner derivation,
where $M$ is a fixed element in $\mathcal{M}$.

In \cite{sakai}, S. Sakai proves that every derivation on von Neumann algebras is an inner derivation.
In \cite{christensen}, E. Christensen shows that every derivation on nest algebras on a Hilbert space $\mathcal H$
is an inner derivation.
For more information on derivations and inner derivations, we refer to \cite{DuZhang,DuZhang1,hejun}.

In \cite{R. Kadison,D. Larson}, R. Kadison and D. Larson introduce the concept of local derivations.
A linear mapping $\delta$ from $\mathcal{A}$ into
$\mathcal{M}$ is called a \emph{local derivation} if for every $A$ in $\mathcal{A}$,
there exists a derivation $\delta_{A}$ (depending on $A$) from $\mathcal{A}$ into $\mathcal{M}$
such that $\delta(A)=\delta_{A}(A)$.

In \cite{R. Kadison}, R. Kadison proves that every continuous local
derivation from a von Neumann algebra into its dual Banach
module is a derivation. In \cite{D. Larson}, D. Larson and A. Sourour prove that
if $X$ is a Banach space, then every local derivation on $B(X)$ is a derivation.
In \cite{B. Johnson}, B. Jonson shows that every local derivation from a $C^*$-algebra
into its Banach bimodule is a derivation. In \cite{HadwinLi1, HadwinLi2}, D. Hadwin and J. Li
characterize local derivations on non self-adjoint operator algebras such as nest algebras and CDCSL algebras.

A linear mapping $\varphi$ from $\mathcal{A}$ into an $\mathcal{A}$-bimodule $\mathcal{M}$ is called a Lie derivation if
$\varphi[A,B]=[\varphi(A),B]+[A,\varphi(B)]$ for each $A$ and $B$ in $\mathcal{A}$, where $[A,B]=AB-BA$ is the usual Lie product.
A Lie derivation $\varphi$ is said to be standard if it can be decomposed as $\varphi=\delta+\tau$, where $\delta$ is a derivation from
$\mathcal{A}$ into $\mathcal{M}$ and $\tau$ is a linear mapping from $\mathcal{A}$ into $\mathcal{Z}(\mathcal{A},\mathcal{M})$ such that $\tau[A,B]=0$ for
each $A$ and $B$ in $\mathcal{A}$, where $\mathcal{Z}(\mathcal{A},\mathcal{M})=\{M\in\mathcal{M}:AM=MA~\mbox{for every}~A~\mbox{in}~\mathcal{A}\}$.

In \cite{Johnson1}, B. Johnson proves that every continuous Lie derivation from a $C^{*}$-algebra into its Banach bimodule is standard.
In \cite{M. Mathieu}, M. Mathieu and A. Villena prove that every Lie derivation on a $C^*$-algebra is standard.
In \cite{W. Cheung}, W. Cheung characterizes Lie derivations on triangular algebras. In \cite{F. Lu3}, F. Lu proves that every Lie derivation on a
completely distributed commutative subspace lattice algebra is standard. In \cite{D. Benkovic}, D. Benkovi$\check{c}$ proves that every Lie derivation
on matrix algebra $M_n(\mathcal A)$ is standard, where $n\geq2$ and $\mathcal A$ is a 2-torsion free unital algebra.

Similar to local derivations, In \cite{L. Chen}, L. Chen, F. Lu and T. Wang introduce the concept of local Lie derivations.
A linear mapping $\varphi$ from $\mathcal{A}$ into
$\mathcal{M}$ is called a \emph{local Lie derivation} if for every $A$ in $\mathcal{A}$,
there exists a Lie derivation $\varphi_{A}$ (depending on $A$) from $\mathcal{A}$ into $\mathcal{M}$
such that $\varphi(A)=\varphi_{A}(A)$.

In \cite{L. Chen}, L. Chen, F. Lu and T. Wang prove that every local Lie derivation on $B(X)$ is a Lie derivation, where $X$ is a Banach space of dimension exceeding 2.
In \cite{L. Chen2},  L. Chen and F. Lu prove that every local Lie derivation on nest algebras is a Lie derivation.
In \cite{liudan}, D. Liu and J. Zhang prove that under certain conditions, every local Lie derivation on triangular algebras is a Lie derivation.
In \cite{hejun1}, J. He, J. Li, G. An and W. Huang prove that every local Lie derivation on some algebras such as finite von Neumann algebras, nest algebras, Jiang-Su algebra and UHF algebras
is a Lie derivation.

Compare with the characterizations of derivations on Banach algebras, investigation of
derivations on  unbounded operator algebras begin much later.

In \cite{Segal}, I. Segal studies the theory of noncommutative integration,
and introduces various classes of non-trivial $*$-algebras of unbounded
operators. In this paper, we mainly consider
the $*$-algebra $S(\mathcal M)$ of all measurable operators
and the $*$-algebra $LS(\mathcal M)$ of all locally measurable operators
affiliated with a von Neumann algebra $\mathcal M$.
In \cite{Segal}, I. Segal shows that the algebraic and topological properties
of the measurable operators algebra $S(\mathcal M)$ are similar to the von Neumann algebra $\mathcal M$.
If $\mathcal M$ is a commutative von Neumann algebra,
then $\mathcal M$ is $*$-isomorphic to the algebra
$L^\infty(\Omega,\Sigma,\mu)$ of all essentially bounded measurable
complex functions on a measure space $(\Omega,\Sigma,\mu)$; and
$S(\mathcal M)$ is $*$-isomorphic to the algebra $L^0(\Omega,\Sigma,\mu)$
of all measurable almost everywhere finite complex-valued functions on
$(\Omega,\Sigma,\mu)$. In \cite{Ber1}, A. Ber, V. Chilin and F. Sukochev show that there exists a derivation on $L^0(0,1)$
is not an inner derivation, and the derivation is discontinuous in the measure topology.
This result means that the properties of derivations on $S(\mathcal M)$
are different from the derivations
on $\mathcal M$.

In \cite{Albeverio1, Albeverio2}, Albeverio, Ayupov and Kudaybergenov study
the properties of derivations on various classes of measurable algebras.
If $\mathcal M$ is a type I von Neumann algebra,
in \cite{Albeverio1}, the authors prove that every derivation
on $LS(\mathcal M)$ is an inner derivation if and only if it is $\mathcal Z(\mathcal M)$ linear;
in \cite{Albeverio2}, the authors give the decomposition form of derivations on $S(\mathcal M)$
and $LS(\mathcal M)$; they also prove that if $\mathcal M$ is a type $\mathrm{I}_\infty$ von Neumann algebra,
then every derivation on $S(\mathcal M)$ or $LS(\mathcal M)$ is an inner derivation.
If $\mathcal M$ is a properly infinite von Neumann algebra,
in \cite{Ber3}, A. Ber, V. Chilin and F. Sukochev prove that
every derivation on $LS(\mathcal M)$ is continuous with respect
to the local measure topology $t(\mathcal M)$; and in \cite{Ber2}, the authors
show that every derivation on $LS(\mathcal M)$ is an inner derivation.
In \cite{Albeverio3}, S. Albeverio and S. Ayupov give a characterization of local derivations
on $S(\mathcal M)$, where $\mathcal M$ is an abelian von Neumann algebra.
In \cite{D. Hadwin}, D. Hadwin and J. Li prove that if $\mathcal M$ is a von Neumann
algebra without abelian direct summands, then every local derivation
on $LS(\mathcal M)$ or $S(\mathcal M)$ is a derivation.
In \cite{V. Chilin}, V. Chilin and I. Juraev  show that every Lie derivation on
$LS(\mathcal M)$ or $S(\mathcal M)$ is standard.

This paper is organized as follows. In Section 2, we recall the definitions of algebras of measurable operators and local measurable operators.

In Section 3, we generalize the Corollary 3.2 in \cite{hejun1} and prove that every local Lie derivation
on von Neumann algebras is a Lie derivation.

In Section 4, we prove that if $\mathcal M$ is a type I von Neumann algebra with an atomic lattice
of projections, then every local Lie derivation on $LS(\mathcal M)$ is a Lie derivation.

\section{Preliminaries}\

Let $\mathcal H$ be a complex Hilbert space
and $B(\mathcal H)$ be the algebra of all bounded linear operators on $\mathcal H$.
Suppose that $\mathcal{M}$ is a von Neumann algebra on
$\mathcal{H}$ and
$\mathcal Z(\mathcal M)=\mathcal M\cap\mathcal M'$
is the center of $\mathcal M$, where
$$\mathcal M'=\{a\in B(\mathcal H):ab=ba~\mathrm{for~every}~b~\mathrm{in}~\mathcal M\}.$$
Denote by
$\mathcal P(\mathcal M)=\{p\in\mathcal M:p=p^*=p^2\}$
the lattice of all projections in $\mathcal M$ and by $\mathcal P_{fin}(\mathcal M)$
the set of all finite projections in $\mathcal M$. For each $p$ and $q$
in $\mathcal P(\mathcal M)$, if we define the
inclusion relation $p\subset q$ by $p\leq q$, then
$\mathcal P(\mathcal M)$ is a complete lattice. Suppose that
$\{p_l\}_{l\in\lambda}$ is a family of projections in $\mathcal M$, we denote
$${\sup_{{l\in\lambda}}}p_l=\overline{\bigcup_{l\in\lambda} p_l\mathcal H}~~\mathrm{and}~~
{\inf_{{l\in\lambda}}}p_l=\bigcap_{l\in\lambda} p_l\mathcal H.$$
If $\{p_l\}_{l\in\lambda}$
is an orthogonal family of projections in $\mathcal M$, then we have that
$${\sup_{{l\in\lambda}}}p_l=\sum_{l\in\lambda}p_l.$$

Let $x$ be a closed densely defined linear operator on $\mathcal H$
with the domain $\mathcal D(x)$, where $\mathcal D(x)$ is a linear subspace of $\mathcal H$.
$x$ is said to be \emph{affiliated} with $\mathcal M$,
denote by $x\eta\mathcal M$, if
$u^*xu=x$ for every unitary element $u$ in $\mathcal{M}'$.

A linear operator affiliated with $\mathcal M$
is said to be \emph{measurable} with respect to
$\mathcal M$, if there exists a sequence $\{p_n\}_{n=1}^{\infty}\subset\mathcal P(\mathcal M)$ such that $p_n\uparrow1$,
$p_n(\mathcal H)\subset\mathcal D(x)$ and $p_n^{\bot}=1-p_n\in \mathcal P_{fin}(\mathcal M)$ for every $n\in\mathbb{N}$,
where $\mathbb{N}$ is the set of all natural numbers. Denote by $S(\mathcal M)$ the set of all measurable operators
affiliated with the von Neumann algebra $\mathcal M$.

A linear operator affiliated with $\mathcal M$
 is said to be \emph{locally measurable} with respect to
$\mathcal M$, if there exists a sequence $\{z_n\}_{n=1}^{\infty}\subset\mathcal P(\mathcal Z(\mathcal M))$
such that $z_n\uparrow1$ and $z_nx\in S(\mathcal M)$ for every $n\in\mathbb{N}$.
Denote by $LS(\mathcal M)$ the set of all locally measurable operators
affiliated with the von Neumann algebra $\mathcal M$.

In \cite{M. Muratov}, Muratov and Chilin prove that $S(\mathcal M)$ and $LS(\mathcal M)$
are both unital $*$-algebras and $\mathcal M\subset S(\mathcal M)\subset LS(\mathcal M)$;
the authors also show that if $\mathcal M$ is a finite von Neumann algebra
or $\mathrm{dim}(\mathcal Z(\mathcal M))<\infty$, then
$S(\mathcal M)=LS(\mathcal M)$; if $\mathcal M$ is a type III von Neumann algebra and
$\mathrm{dim}(\mathcal Z(\mathcal M))=\infty$, then $S(\mathcal M)=\mathcal M$ and
$LS(\mathcal M)\neq\mathcal M$.

\section{Local Lie derivations on von Neumann algebras}\

In this section, we consider local Lie derivations on von Neumann algebras. To prove
our main theorem, we need the following lemma.

\begin{lemma}\label{01}
Let $\mathcal{A}_1$ and $\mathcal{A}_2$ be two unital algebras and $\mathcal{A}=\mathcal{A}_1\bigoplus\mathcal{A}_2$.
If the following five conditions hold:\\
$(1)$ each Lie derivation on $\mathcal{A}$ is standard;\\
$(2)$ each derivation on $\mathcal{A}$ is inner;\\
$(3)$ each local derivation on $\mathcal{A}$ is a derivation;\\
$(4)$ $\mathcal{Z}(\mathcal{A}_1)\cap [\mathcal{A}_1,\mathcal{A}_1]=\{0\}$;\\
$(5)$ $\mathcal{A}_2=[\mathcal{A}_2,\mathcal{A}_2]$,\\
then every local Lie derivation on $\mathcal{A}$ is a Lie derivation.
\end{lemma}
\begin{proof}
Denote the units of $\mathcal{A}$, $\mathcal{A}_1$ and $\mathcal{A}_2$ by $I$, $P$ and $Q$, respectively.
For each $A$ in $\mathcal A$, we have that $A=PA+QA=A_1+A_2$, where $A_i\in\mathcal{A}_i,~i=1,2$.

In the following we suppose that $\varphi$ is a local Lie derivation on $\mathcal{A}$.

By the definition of local Lie derivation, we know that for every $A_1$ in $\mathcal{A}_1$,
there exists a Lie derivation $\varphi_{A_1}$ on $\mathcal{A}$ such that $\varphi(A_1)=\varphi_{A_1}(A_1)$.
Since $\varphi_{A_1}$ is standard and each derivation on $\mathcal{A}$ is inner, we can obtain that
$$\varphi(A_1)=\varphi_{A_1}(A_1)=\delta_{A_1}(A_1)+\tau_{A_1}(A_1)=[A_1,T_{A_1}]+P\tau_{A_1}(A_1)+Q\tau_{A_1}(A_1),$$
where $\delta_{A_1}$ is a derivation on $\mathcal{A}$, $T_{A_1}$ is an element in $\mathcal{A}$, and $\tau_{A_1}$ is a linear mapping from
$\mathcal{A}$ into $\mathcal{Z}(\mathcal{A})$ such that $\tau_{A_1}([\mathcal{A},\mathcal{A}])=0$.

It means that $\varphi$ has a decomposition at $A_1$. Next we show that the decomposition at $A_1$
is unique. Assume there is another decomposition at $A_1$, that is
$$\varphi(A_1)=\varphi^{'}_{A_1}(A_1)=\delta^{'}_{A_1}(A_1)+\tau^{'}_{A_1}(A_1)=[A_1,T^{'}_{A_1}]+P\tau^{'}_{A_1}(A_1)+Q\tau^{'}_{A_1}(A_1),$$
where $\delta^{'}_{A_1}$ is a derivation on $\mathcal{A}$, $T^{'}_{A_1}$ is an element in $\mathcal{A}$ and $\tau^{'}_{A_1}$ is a linear mapping from
$\mathcal{A}$ into $\mathcal{Z}(\mathcal{A})$ such that $\tau^{'}_{A_1}([\mathcal{A},\mathcal{A}])=0$.

Then we have that
$$[A_1,T_{A_1}]+P\tau_{A_1}(A_1)+Q\tau_{A_1}(A_1)=[A_1,T^{'}_{A_1}]+P\tau^{'}_{A_1}(A_1)+Q\tau^{'}_{A_1}(A_1).$$
Thus
$$[A_1,T_{A_1}]-[A_1,T^{'}_{A_1}]=P\tau^{'}_{A_1}(A_1)-P\tau_{A_1}(A_1)+Q\tau^{'}_{A_1}(A_1)-Q\tau_{A_1}(A_1).$$

Since $[A_1,T_{A_1}]-[A_1,T^{'}_{A_1}]$ and $P\tau^{'}_{A_1}(A_1)-P\tau_{A_1}(A_1)\in\mathcal{A}_1$ belong to $\mathcal{A}_1$,
and $Q\tau^{'}_{A_1}(A_1)-Q\tau_{A_1}(A_1)$ belongs to $\mathcal{A}_2$,
we have that $Q\tau^{'}_{A_1}(A_1)-Q\tau_{A_1}(A_1)=0$. Moreover, we can obtain that
$$[A_1,T_{A_1}]-[A_1,T^{'}_{A_1}]=[A_1,PT_{A_1}]-[A_1,PT^{'}_{A_1}]\in[\mathcal{A}_1,\mathcal{A}_1],$$
and
$$P\tau^{'}_{A_1}(A_1)-P\tau_{A_1}(A_1)\in\mathcal{Z}(\mathcal{A}_1).$$

By condition (4), it follows that
$[A_1,T_{A_1}]-[A_1,T^{'}_{A_1}]=P\tau^{'}_{A_1}(A_1)-P\tau_{A_1}(A_1)=0.$
It implies that $\delta_{A_1}(A_1)=\delta^{'}_{A_1}(A_1)$ and $\tau_{A_1}(A_1)=\tau^{'}_{A_1}(A_1)$.
Hence the decomposition is unique.

Now we have $\varphi|_{\mathcal{A}_1}=\delta_1+\tau_1$, where $\delta_1$ is a mapping from $\mathcal{A}_1$ into $\mathcal{A}_1$
such that
$\delta_1(A_1)=[A_1,S_{A_1}]$ for some element $S_{A_1}$ in $\mathcal{A}_1$, and $\tau_1$ is a
mapping from $\mathcal{A}_1$ into $\mathcal{Z}(\mathcal{A})$ such that $\tau_1([\mathcal{A}_1,\mathcal{A}_1])=0$.

Next we prove that $\delta_1$ and $\tau_1$ are linear mappings. For each $A_1$ and $B_1$ in $\mathcal{A}_1$, we have that
$$\varphi(A_1)=\delta_1(A_1)+\tau_1(A_1)=[A_1,S_{A_1}]+\tau_1(A_1),$$
$$\varphi(B_1)=\delta_1(B_1)+\tau_1(B_1)=[B_1,S_{B_1}]+\tau_1(B_1),$$
and
$$\varphi(A_1+B_1)=\delta_1(A_1+B_1)+\tau_1(A_1+B_1)=[A_1+B_1,S_{A_1+B_1}]+\tau_1(A_1+B_1).$$
Since $\varphi$ is additive, through a discussion similar to that before, it implies that
$$[A_1+B_1,S_{A_1+B_1}]=[A_1,S_{A_1}]+[B_1,S_{B_1}]$$
and
$$\tau_1(A_1+B_1)=\tau_1(A_1)+\tau_1(B_1).$$
It means that $\delta_1$ and $\tau_1$ are additive mappings. Using the same technique,
we can prove that $\delta_1$ and $\tau_1$ are homogeneous. Hence $\delta_1$ and $\tau_1$ are linear mappings.

For every $A_2$ in $\mathcal{A}_2$, we have that
$$\varphi(A_2)=\varphi_{A_2}(A_2)=\delta_{A_2}(A_2)+\tau_{A_2}(A_2)=[A_2,T_{A_2}]+\tau_{A_2}(A_2),$$
where $\delta_{A_2}$ is a derivation on $\mathcal{A}$, $T_{A_2}$ is an element in $\mathcal{A}$ and $\tau_{A_2}$ is a linear mapping from
$\mathcal{A}$ into $\mathcal{Z}(\mathcal{A})$ such that $\tau_{A_2}([\mathcal{A},\mathcal{A}])=0$.
By condition (5), we have that
$\tau_{A_2}(A_2)=0$. Thus $\varphi(A_2)=[A_2,T_{A_2}]=[A_2,QT_{A_2}].$

Let $\varphi|_{\mathcal{A}_2}=\delta_2$. Then we have $\delta_2(A_2)=[A_2,S_{A_2}]$ for some element $S_{A_2}$ in $\mathcal{A}_2$.
And obviously, $\delta_2$ is linear.

Define two linear mappings as follows:
$$\delta(A)=\delta_1(A_1)+\delta_2(A_2),~~
\tau(A)=\tau_1(A_1),$$
for all $A=A_1+A_2\in\mathcal{A}$.
By the previous discussion, $\tau$ is a linear mapping from
$\mathcal{A}$ into $\mathcal{Z}(\mathcal{A})$ such that $\tau([\mathcal{A},\mathcal{A}])=0$.
In addition,
$$\delta(A)=\delta_1(A_1)+\delta_2(A_2)=[A_1,S_{A_1}]+[A_2,S_{A_2}]=[A_1+A_2,S_{A_1}+S_{A_2}]=[A,S_{A_1}+S_{A_2}].$$
It means that $\delta$ is a local derivation. By condition (3), $\delta$ is a derivation.
Notice that
$$\varphi(A)=\varphi(A_1)+\varphi(A_2)=\delta_1(A_1)+\tau_1(A_1)+\delta_2(A_2)=\delta(A)+\tau(A).$$
Hence $\varphi$ is a standard Lie derivation.
\end{proof}

By Lemma 3.1, we have the following result.

\begin{theorem}\label{02}
Every local Lie derivation on a von Neumann algebra is a Lie derivation.
\end{theorem}
\begin{proof}
Let $\mathcal A$ be a von Neumann algebra. It is well known that $\mathcal{A}=\mathcal{A}_1\bigoplus\mathcal{A}_2$,
where $\mathcal{A}_1$ is a finite von Neumann algebra, and $\mathcal{A}_2$ is a proper infinite von Neumann algebra.

By \cite[Theorem 1.1]{M. Mathieu}, we know that every Lie derivation on $\mathcal{A}$ is standard,
by \cite[Theorem 1]{sakai}, we know that every derivation on $\mathcal{A}$ is inner, and by
\cite[Theorem 5.3]{B. Johnson}, we know that every local derivation on $\mathcal{A}$ is a derivation.
Since $\mathcal{A}_2$ is a proper infinite von Neumann algebra, we known that $\mathcal{A}_2=[\mathcal{A}_2,\mathcal{A}_2]$(see in \cite{sunouchi}).

Hence it is sufficient to prove that $\mathcal{Z}(\mathcal{A}_1)\cap [\mathcal{A}_1,\mathcal{A}_1]=\{0\}$.
Since $\mathcal{A}_1$ is finite and by \cite[Theorem 8.2.8]{R.Kadison J. Ringrose}, it follows that there is a center-valued trace $\tau$ on $\mathcal{A}_1$
such that $\tau(Z)=Z$ for every $Z$ in $\mathcal{Z}(\mathcal{A}_1)$ and $\tau([A,B])=0$ for each $A$ and $B$ in $\mathcal{A}_1$.
Suppose that $A\in\mathcal{Z}(\mathcal{A}_1)\cap [\mathcal{A}_1,\mathcal{A}_1]$,
then we have that $\tau(A)=A$ and $\tau(A)=0$. it implies that $A=0$.

By Lemma \ref{01}, we know that every local Lie derivation on a von Neumann algebra is a Lie derivation.
\end{proof}

\section{Local Lie derivations on algebras of locally measurable operators}\

In this section, we mainly consider local Lie derivations on algebras
of all locally measurable operators affiliated with a type I von Neumann algebra.
To prove the main result, we need the following lemmas.

\begin{lemma}\label{03}
Suppose that $\mathcal{A}$ is a commutative unital algebra and $\mathcal{J}=M_n(\mathcal{A})$.
Then $\mathcal{Z}(\mathcal{J})\cap[\mathcal{J},\mathcal{J}]=\{0\}$
\end{lemma}
\begin{proof}
Let $\{e_{i,j}\}_{i,j=1}^{n}$ be the system of matrix units in $M_n(\mathcal{A})$.
Then for every element $A$ in $\mathcal{J}$, we have that $A=\sum_{i,j=1}^{n}a_{ij}e_{ij}$,
where $a_{ij}\in\mathcal{A}$.

Define a linear mapping $\tau$ from $\mathcal{J}$ into $\mathcal{A}$ by
$\tau(A)=\sum_{i=1}^{n}a_{ii}$ for every $A=\sum_{i,j=1}^{n}a_{ij}e_{ij}\in\mathcal{J}$.
Since $\mathcal{A}$ is  commutative, it is not difficult to verify that $\tau([A,B])=0$ for each $A$ and $B$ in $\mathcal{J}$.

It should be noticed that $\mathcal{Z}(\mathcal{J})=\{A:~A=\sum_{i=1}^{n}ae_{ii},~a\in\mathcal{A}\}$.
Suppose that $A=\sum_{i=1}^{n}ae_{ii}$ is an element in $\mathcal{Z}(\mathcal{J})\cap[\mathcal{J},\mathcal{J}]$,
then by the definition of $\tau$, we have that $\tau(A)=na$ and $\tau(A)=0$.
It implies that $A=0$.
\end{proof}

\begin{lemma}\label{04}
Suppose that $\mathcal{A}=\prod_{i\in\Lambda}\mathcal{A}_i$.
If $\mathcal{Z}(\mathcal{A}_i)\cap [\mathcal{A}_i,\mathcal{A}_i]=\{0\}$ for every $i\in\Lambda$,\\
then we have that $\mathcal{Z}(\mathcal{A})\cap [\mathcal{A},\mathcal{A}]=\{0\}$.
\end{lemma}
\begin{proof}
Let $A=\{a_i\}_{i\in\Lambda}$ be an element in $\mathcal{Z}(\mathcal{A})\cap [\mathcal{A},\mathcal{A}]$.
Then for every $i\in\Lambda$, we have that $a_i\in\mathcal{Z}(\mathcal{A}_i)\cap [\mathcal{A}_i,\mathcal{A}_i]$.
By assumption, it follows that $a_i=0$. Hence $A=0$.
\end{proof}

\begin{lemma}\label{05}
Suppose that $\mathcal{M}$ is a type $\mathrm{I}_{\infty}$ von Neumann algebra. Then $LS(\mathcal{M})=[LS(\mathcal{M}),LS(\mathcal{M})]$.
\end{lemma}
\begin{proof}
By \cite{M. Muratov1}, we know that for every $x$ in $LS(\mathcal{M})$, there exists a sequence $\{z_n\}$ of mutually orthogonal central projections in $\mathcal{M}$ with $\sum_{n=1}^{\infty}z_n=I$,
such that $x=\sum_{n=1}^{\infty}z_nx$, and $z_nx\in\mathcal{M}$ for every $n\in\mathbb{N}$.
Since $\mathcal{M}$ is a proper infinite von Neumann algebra, it is well known that $\mathcal{M}=[\mathcal{M},\mathcal{M}]$.
Thus we have that $z_nx=\sum_{i=1}^{k}[a^{n}_i,b^{n}_i]$, where $a^{n}_i,b^{n}_i\in\mathcal{M}$ for each $n$ and $i$.

Set $s_i=\sum_{n=1}^{\infty}z_na^{n}_i$ and $t_i=\sum_{n=1}^{\infty}z_nb^{n}_i$.
By the definition of locally measurable operators, it is easy to show that $s_i$ and $t_i$ are two elements in $LS(\mathcal{M})$.

Since that $\{z_n\}$ are mutually orthogonal central projections, we can obtain that
$$[s_i,t_i]=[\sum_{n=1}^{\infty}z_na^{n}_i,\sum_{n=1}^{\infty}z_nb^{n}_i]=\sum_{n=1}^{\infty}z_n[a^{n}_i,b^{n}_i],$$
moreover, we have that
$$\sum_{i=1}^{k}[s_i,t_i]=\sum_{i=1}^{k}\sum_{n=1}^{\infty}z_n[a^{n}_i,b^{n}_i]=\sum_{n=1}^{\infty}z_n(\sum_{i=1}^{k}[a^{n}_i,b^{n}_i])
=\sum_{n=1}^{\infty}z_nx=x.$$
It follows that $x\in[LS(\mathcal{M}),LS(\mathcal{M})]$.
\end{proof}

In the following we show the main result of this section.

\begin{theorem}\label{06}
Suppose that $\mathcal{M}$ is a type $\mathrm{I}$ von Neumann algebra with an atomic lattice of projections.
Then every local Lie derivation from $LS(\mathcal{M})$ into itself is a Lie derivation.
\end{theorem}
\begin{proof}
By \cite[Theorem 6.5.2]{R.Kadison J. Ringrose}, we know that $\mathcal{M}=\mathcal{M}_1\bigoplus\mathcal{M}_2$,
where $\mathcal{M}_1$ is a type $\mathrm{I}_{finite}$ von Neumann algebra
and $\mathcal{M}_2$ is a type $\mathrm{I}_{\infty}$ von Neumann algebra.
Hence by \cite[Proposition 1.1]{Albeverio2}, we have that $LS(\mathcal{M})\cong LS(\mathcal{M}_1)\bigoplus LS(\mathcal{M}_2)$.

In the following we will verify the conditions (1) to (5) in Lemma \ref{01} one by one.

By \cite[Theorem 1]{V. Chilin}, we know that every Lie derivation on $LS(\mathcal{M})$ is standard;
by \cite[Corollary 5,12]{Albeverio2}, we know that every derivation on $LS(\mathcal{M})$ is inner for a von Neumann
algebra with atomic lattice of projections.

It is proved in \cite{HadwinLi1} that every local derivation on $LS(\mathcal{M})$ is a derivation
for a von Neumann algebra without abelian direct summands.
While for an abelian von Neumann algebra with atomic lattice of projections,
by \cite[Theorem 3.8]{Albeverio3} we know that every local derivation on $LS(\mathcal{M})$ is a derivation.
Associated the two results, we can obtain each local derivation on $LS(\mathcal{M})$ is a derivation
for a von Neumann algebra with atomic lattice of projections.

Since $\mathcal{M}_1$ is a type $\mathrm{I}_{finite}$ von Neumann algebra,
we know that $\mathcal{M}_1=\bigoplus_{n=1}^{\infty}\mathcal{A}_n$,
where each $\mathcal{A}_n$ is a homogenous type $\mathrm{I}_n$ von Neumann algebra.
Hence $LS(\mathcal{M}_1)\cong\prod_{n=1}^{\infty}LS(\mathcal{A}_n)$.
Since $\mathcal{A}_n$ is a homogenous type $\mathrm{I}_n$ von Neumann algebra,
by \cite{Albeverio2} we know that $LS(\mathcal{A}_n)\cong M_n(\mathcal{Z}(LS(\mathcal{A}_n)))$.
By Lemmas \ref{03} and \ref{04}, we know that the condition (4) in Lemma \ref{01} holds.
And by Lemma \ref{05}, the condition (5) in Lemma \ref{01} holds.
\end{proof}

\bibliographystyle{amsplain}

\begin{thebibliography}{99}

\bibitem{Albeverio1} S. Albeverio, S. Ayupov, K. Kudaybergenov.
Derivations on the algebra of measurable operators affiliated with a type I von Neumann algebra.
Siberian Adv. Math., 2008, 18: 86-94.

\bibitem{Albeverio2} S. Albeverio, S. Ayupov, K. Kudaybergenov.
Structure of derivations on various algebras of measurable operators for type I von Neumann algebras.
J. Func. Anal., 2009, 256: 2917-2943.

\bibitem{Albeverio3} S. Albeverio, S. Ayupov, K. Kudaybergenov, B. Nurjanov. Local derivations on algebras of
measurable operators. Comm. In Contem. Math., 2011, 13: 643--657.

\bibitem{D. Benkovic} D. Benkovi$\check{c}$. Lie triple derivations of unital algebras with idempotents.
Linear Multilinear Algebra, 2015, 63: 141--165.

\bibitem{Ber1} A. Ber, V. Chilin, F. Sukochev.
Non-trivial derivation on commutative regular algebras.
Extracta Math., 2006, 21: 107--147.

\bibitem{Ber3} A. Ber, V. Chilin, F. Sukochev.
Continuity of derivations of algebras of locally measurable operators.
Integr. Equ. Oper. Theory, 2013, 75: 527--557.

\bibitem{Ber2} A. Ber, V. Chilin, F. Sukochev.
Continuous derivations on algebras of locally measurable operators are inner.
Proc. London Math. Soc., 2014, 109: 65--89.

\bibitem{M. Bresar} M. Bre$\check{s}$ar, E. Kissin, S. Shulman,
Lie ideals: from pure algebra to C*-algebras,
J. reine angew. Math.,
2008, 623: 73--121.

\bibitem{V. Chilin} V. Chilin, I. Juraev. Lie derivations on the algebras of locally measurable operators. 2016,
arXiv: 1608. 03996v1.

\bibitem{L. Chen} L. Chen, F. Lu, T. Wang.
Local and 2-local Lie derivations of operator algebras on Banach spaces.
Integr. Equ. Oper. Theory,
2013, 77: 109--121.

\bibitem{L. Chen2} L. Chen, F. Lu. Local Lie derivations of nest algebras. Linear Algebra Appl.,
2015, 475: 62--72.

\bibitem{W. Cheung} W. Cheung. Lie derivations of triangular algebra.
Linear Multilinear Algebra, 2003, 512: 299--310.

\bibitem{christensen} E. Christensen. Derivations of nest algebras. Math. Ann., 1977, 229: 155--161.

\bibitem{DuZhang} H. Du, J. Zhang. Derivations on nest-subalgebras of von Neumann algebras. Chin. Ann. Math., 1996, A 17: 467--474.

\bibitem{DuZhang1} H. Du, J. Zhang. Derivations on nest-subalgebras of von Neumann algebras II. Acta. Math., 1997, A 40: 357--362.

\bibitem{D. Hadwin} D. Hadwin, J. Li, Q. Li, X. Ma. Local derivations on rings containing a von Neumann
algebra and a question of Kadison. 2013, arXiv:1311.0030v1.

\bibitem{HadwinLi1} D. Hadwin, J. Li. Local derivations and local automorphisms on some algebras. J. Oper. Theory, 2008, 60: 29--44.

\bibitem{HadwinLi2} D. Hadwin, J. Li. Local derivations and local automorphisms. J. Math. Anal. Appl., 2004, 290: 702--714.

\bibitem{hejun} J. He, J. Li, D. Zhao. Derivations, Local and 2-Local derivations on some algebras of
operators on Hilbert C*-bodules. Mediterr. J. Math., 2017, 14: article 230.

\bibitem{hejun1} J. He, J. Li, G. An, W. Huang. Characterizations of 2-local derivations and local Lie
derivations on some algebras. Sib. Math. J., 2017, accepted.

\bibitem{B. Johnson} B. Johnson. Local derivations on $C*$-algebras are derivations.
Trans. Amer. Math. Soc., 2001, 353: 313--325.

\bibitem{Johnson1} B. Johnson. Symmetric amenability and the nonexistence of Lie and Jordan derivations. Math. Proc. Cambd. Philos. Soc., 1996, 120: 455--473.

\bibitem{R. Kadison} R. Kadison. Local derivations. J. Algebra, 1990, 130: 494--509.

\bibitem{R.Kadison J. Ringrose} R. Kadison, J. Ringrose, Fundamentals of the Theory of Operator Algebras,
I, Pure Appl. Math. 100, Academic Press, New York, 1983.

\bibitem{D. Larson} D. Larson, A. Sourour. Local derivations and local automorphisms.
Proc. Sympos. Pure Math., 1990, 51: 187--194.

\bibitem{liudan} D. Liu, J. Zhang. Local Lie derivations on certain operator algebras. Ann. Funct. Anal., 2017: 270--280.


\bibitem{F. Lu3} F. Lu. Lie derivation of certain CSL algebras. Israel J. Math.,
2006, 155: 149--156.

\bibitem{M. Mathieu} M. Mathieu, A. Villena. The structure of Lie derivations on $C^*$-algebras.
J. Funct. Anal., 2003, 202: 504--525.

\bibitem{M. Muratov} M. Muratov, V. Chilin. Algebras of measurable and locally measurable operators.
Kyiv, Pratse In-ty matematiki NAN ukraini., 2007, 69: 390 pp, (Russian).

\bibitem{M. Muratov1} M. Muratov, V. Chilin. Central extensions of *-algebras of measurable operators.
Reports of the National Academy of Science of Ukraine, 2009, 7: 24--28. (Russian).

\bibitem{sakai} S. Sakai. Derivations of $W^{*}$-algebras. Ann. Math., 1966, 83: 273--279.

\bibitem{Segal} I. Segal. A non-commutative extension of abstract integration. Ann, Math., 1953, 57: 401--457.

\bibitem{sunouchi} H. Sunouchi. Infinite Lie rings. Tohoku Math. J., 1956, 8: 291--307.

\end{thebibliography}

\end{document}